\documentclass[11pt]{article}
\usepackage[title]{appendix}
\usepackage[colorlinks]{hyperref}
\usepackage{epsfig, amsmath, amssymb, amsthm, times, esint, mathtools}
\usepackage{color}

\newtheorem {thm}{Theorem}[section]
\newtheorem {proposition}[thm]{Proposition}

 \newtheorem {lem}[thm]{Lemma}

 \theoremstyle{defintion}
 \newtheorem {df}[thm]{Definition}
 \theoremstyle{assumption}
 
 \theoremstyle{remark}
 \newtheorem{rem}[thm]{Remark}
 \theoremstyle{example}

\def\P{\operatorname{\mathbb P}}

\def\R{{\mathbb R}}
\def\N{{\mathbb N}}
\def\Z{{\mathbb Z}}

\def\cB{\mathcal{B}}
\def\cH{\mathcal{H}}
\def\cM{\mathcal{M}}

\def\cW{\mathcal{W}}
\def\cL{\mathcal{L}}

\def\lbl{\label}
\def\be{\begin{equation}}
\def\ee{\end{equation}}
\def\1{\operatorname{\mathbf 1}}

\newcommand{\iu}{{i\mkern1mu}}

\def\sup{{\rm sup}}

\def\W{\mathcal W}
\def\bbR{\mathbb R}

\def\bbZ{\mathbb Z}
\def\bbT{\mathbb T}
\def\bb1{\mathds 1}

\def\bP{\mathbf{P}}

\newcommand{\red}[1]{{\color{black} #1}}

\title{The Large Deviation Principle for $W$-random spectral measures}
\date{\today}
\author{ Mahya Ghandehari\thanks{Department of Mathematical Sciences, University of Delaware,
Newark, DE 19716, USA,
    {\tt mahya@udel.edu}}\;
  and 
Georgi S. Medvedev\thanks{Department of Mathematics, 
Drexel University, 3141 Chestnut St., Philadelphia, PA 19104, USA, 
{\tt medvedev@drexel.edu}}
}

\begin{document}
\maketitle
\begin{abstract}
  The $W$-random graphs provide a flexible framework for modeling large random
  networks. Using the Large Deviation Principle (LDP) for $W$-random graphs from \cite{DupMed22},
  we prove the LDP for the corresponding class of random symmetric Hilbert-Schmidt
  integral operators. Our main result describes how the eigenvalues and the eigenspaces
  of the integral operator are affected by the large deviations in the underlying random graphon.
  To prove the LDP, we demonstrate continuous dependence
  of the spectral measures associated with integral operators
  on the underlying graphons and use the Contraction Principle. To illustrate
  our results, we obtain leading order asymptotics of the eigenvalues of the integral 
  operators corresponding to certain random graph sequences. These examples
  suggest several representative scenarios of how the eigenvalues and the
  eigenspaces of the integral operators are affected by large deviations. 
  Potential implications of these observations for bifurcation analysis of Dynamical Systems and Graph Signal
  Processing are indicated.

\end{abstract}

\section{Introduction}
\setcounter{equation}{0}
The $W$-random graph model, introduced 
% by Lovasz and Szegedy 
in \cite{LovSze06}, 
 provides a flexible framework for modeling large networks.
This model encompasses many network topologies common in applications, such as
 Erd\H{o}s-R{\' e}nyi, small-world, and power--law.
$W$-random graphs have been used in a variety of applications including 
interacting
particle systems  in statistical physics \cite{OliRei20, DupMed22}, coupled dynamical systems
\cite{Med14b, KVMed18},
mean--field games \cite{CaiHua21, CaiHo22}, signal processing on graphs \cite{GhaJan22, Ruiz-2021,Morency2021}, and neural networks \cite{neural-networks1,Neural-network2}, 
to mention a few.

There are several variants of the $W$-random graph model. In this paper, we will stick with the one
used in \cite{Med19}. Specifically, let $W$ be a graphon, that is, 
$W:\,[0,1]^2\rightarrow [0,1]$ is a symmetric measurable function. We define $\Gamma_W^n$,
a random graph on $n$ nodes with the adjacency matrix $(W_{ij}^n)_{1\leq i,j\leq n}$ such that
\begin{align}\lbl{def-Wn-1}
  \P\left(W^n_{ij}=1\right) =w^n_{ij},& \qquad \P\left(W^n_{ij}=0\right) =1-w^n_{ij}, \qquad 1\le i<j\le n,\\
  \lbl{def-Wn-2}
  W^n_{ii}=0, & \qquad W^n_{ij}= W_{ji}^n, 
\end{align}
where
$$
w^n_{ij}\doteq n^2\int_{Q_{ij}^n} W(x,y)\, dxdy,\quad  \mbox{ and }\quad Q_{ij}^n\doteq \left[\frac{i-1}{n},\frac{i}{n}\right]\times\left[\frac{j-1}{n},\frac{j}{n}\right].
$$
Graphs defined by \eqref{def-Wn-1} and \eqref{def-Wn-2} are called $W$-random graphs.
 We denote the $W$-random graph model by $\mathbb{G}(n,W)$.

 A $W$-random graph $\Gamma_W^n$ can be represented by a 0/1-valued step function on the unit square
\be\lbl{Wn}
W^n=\sum_{i,j=1}^n W^n_{ij} \1_{Q^n_{ij}},
\ee
where $(W^n_{ij},\; 1\le i<j\le n)$ are independent random variables defined by
\eqref{def-Wn-1}. 
%Functions $W^n$ and $W$ representing a graph sequence and its limit respectively are called graphons.
Random graphons $W^n$ converge to the graphon $W$ almost surely (cf.~\cite{DupMed22}).
The convergence is with respect to the cut norm
\begin{equation}\label{cut-norm}
  \|W\|_\Box=\sup_{f,g} \left|\int_{[0,1]^2} f(x)W(x,y)g(y)\, dxdy\right|,
  \end{equation}
where the supremum is taken over all measurable $f,g:~[0,1]\to [-1,1]$. 
\footnote{In graph limit theory, the cut norm of a graphon is usually defined as  $\|W\|_\Box=\sup_{A,B} \left|\int_{A\times B} W(x,y)\, dxdy\right|$, where the supremum is taken over all measurable subsets of $[0,1]$. It is easy to see that the two norms are equivalent.}

The cut norm convergence of $\{W^n\}$ has strong implications for the
spectral properties of the integral operators
\begin{equation}\label{int-operator}
K_{W^n}:L^2[0,1]\to L^2[0,1],\quad (K_{W^n}f)(x)\doteq \int W^n(x,y)f(y)dy.
\end{equation}
As shown in \cite{Sze-2011}, the eigenvalues of $K_{W^n}$ converge to the corresponding
eigenvalues of  $K_{W}$. Moreover, the corresponding eigenspaces converge as well in the sense
that will be explained below (see also \cite{Sze-2011}).

The almost sure convergence of $W^n$ to $W$ can be interpreted as a law of large numbers.
Remarkably,  $\{W^n\}$  also admits a Large Deviation Principle (LDP) \cite{DupMed22}.
The LDP can be used to characterize the most likely realizations of $\{W^n\}$  in case
of the rare events. For instance, the LDP was used in \cite{DupMed22} to give an explicit
characterization of the $W$-random graph sequence in case of large deviations of the number of edges
from its expected value. This information is important in applications as
it allows to model rare events effectively. In this paper, we prove an LDP for spectral
measures associated with integral operators $K_{W^n}$. Our main result shows how the
eigenvalues and the eigenspaces of $K_{W^n}$ behave under large deviations.
In applications, the spectrum of $K_{W^n}$ is used to determine important properties such as stability of dynamical systems, expansion in graphs, or frequency of oscillations
in mechanical systems to name a few. Thus, it is important to know how the eigenvalues and
the eigenspaces are affected by large deviations.

The LDP for spectral measures is derived from the LDP for $W$-random graphs proved in \cite{DupMed22}
via the Contraction Principle \cite{Dembo-Zeitouni}. To use the Contraction Principle one needs to
show the continuous correspondence between the two spaces:  the space of graphons and the corresponding space of spectral measures. 
This depends on the choice of topologies in these spaces. Previous studies
of large deviations for random graphs suggest that the cut norm topology is the right choice
for the space of graphons \cite{ChaVar11, DupMed22, Varadhan-topology}. 
The use of the cut norm topology in the context of large deviations is advantageous for a specific reason: after appropriate identification of graphons, cut norm makes the (quotient) space of graphons compact (thanks to Szemer{\' e}di's Lemma \cite{LovSze07}).
For the space of associated spectral measures we adapt the notion of vague convergence \cite{Chung-Probability}
and show continuous dependence of the spectral measure on the corresponding graphon. This
is the main mathematical result of this work. It implies the LDP for $W$-random spectral measures.

After reviewing some preliminaries from the theory of graph limits  in \S~\ref{sec.basic}, we define
the space of spectral measures equipped with vague topology in \S~\ref{sec.vague}. In the theory
of graph limits, one works with the space of unlabeled graphons, i.e., the space of functions where
functions related via measure-preserving transformations are identified \cite{LovaszAMS}.
In this work, we need to extend this construction to spaces of 
bounded linear operators and spectral measures. This is done in \S~\ref{sec.equivalence}, which concludes
the description of the spaces used in this work. In Section~\ref{sec.continuous} we prove continuous
dependence of spectral measures on the underlying graphons. This implies the LDP for $W$-random
spectral measures presented in Sections~\ref{sec.ldp}. Here, we review the LDP for $W$-random
graphs from \cite{DupMed22} in \S~\ref{sec.ldp-old} and derive the LDP for $W$-random spectral measures
in \S~\ref{sec.ldp-new}. To illustrate the LDP for spectral measures, we turn to concrete examples.
Using the results from \cite{DupMed22}, for certain sequences of $W$-random graphs we work out
leading order asymptotic correction for the eigenvalues of $K_{W^n}$ conditioned on large deviations
for the edge counts in $\Gamma_{W^n}$. This examples suggest possible scenarios of how the eigenvalues
of $K_{W^n}$ can be affected by large deviations. The implications of these scenarios for the analysis of Dynamical Systems
and Graph Signal Processing are outlined in Section~\ref{sec.discuss}.

\section{Graphons, kernels, and spectral measures}\label{sec.graphons}
\setcounter{equation}{0}
\subsection{Basic spaces} \label{sec.basic}
In this section, we review the space of graphons \cite{LovaszAMS, Janson2013}
and introduce several spaces that are derived from it.

Let $\cW$ be the space of bounded symmetric measurable functions on the unit square:
$$
\cW=\left\{W\in L^\infty\left([0,1]^2\right):\quad W(x,y)=W(y,x)\;\mbox{a.e.}\right\}.
$$
The elements of $\cW$ are called \textit{kernels}. Denote
$$
\cW_0=\left\{W\in\cW:\quad 0\le W\le 1\;\mbox{a.e.}\right\}.
$$
The elements of $\cW_0$ are called \textit{graphons}.
The space $\cW$ is equipped with the cut norm \eqref{cut-norm}.

%%%
To every $W\in \cW$, there corresponds a linear operator on the Hilbert space $\cH\doteq L^2([0,1])$ defined as 
$$K_W(f)(x)\doteq \int_0^1 W(x,y)f(y)\, dy, \mbox{ for } f\in L^2([0,1]), x\in [0,1].$$ 
Since $W$ is real-valued, bounded, and symmetric, the  operator $K_W$ is compact and self-adjoint  \cite{Young-Hilbert}. 
Thus, the spectrum of $K_W$ is a countable subset of ${\mathbb R}$, with the only possible accumulation point at 0. 
The point spectrum of $K_W$, denoted by $\sigma_p(K_W)$, is the collection of all eigenvalues of $K_W$. 
Since $K_W$ is a compact operator, every nonzero element of the spectrum belongs to $\sigma_p(K_W)$. 
The spectrum of a graphon $W\in\W_0$ is defined to be the spectrum of the corresponding operator $K_W$. We label this spectrum as follows:
\be\lbl{eigW}
1\ge \lambda_{1}\ge \lambda_{2}\ge \dots\ge \lambda_0=0\ge \dots \ge \lambda_{-2}\ge \lambda_{-1}\ge -1.
\ee
In the above list, repeated eigenvalues are listed separately. Since $W\in\cW_0$,  the eigenvalues do not exceed $1$ in the absolute
value.  Further, if the sequence of nonnegative eigenvalues is finite, by default, it is  extended by adding an infinite sequence
of zeros at the end. We deal similarly with the negative eigenvalues, so that $\lambda_j$ is well-defined for every
$j\in\dot\Z\doteq\Z\setminus\{0\}$. 
\footnote{This arrangement is important for our discussion of the convergence of spectra for converging sequences of dense graphs.}
If $K_W$ has $0$ as an eigenvalue, we will denote it as $\lambda_0$ in the above sequence. 
The corresponding sequence of normalized
eigenvectors 
$\{f_i\}_{i\in \Z\,  \& \, \lambda_i\in\sigma_p(K_W)}$ forms an orthonormal basis for  $\cH$; the subset $\{f_i\}_{i\in \dot\Z\,  \& \, \lambda_i\in\sigma_p(K_W)}$ forms an orthonormal basis for $\ker(K_W)^\perp$.  
Clearly, such bases are not unique.

Having fixed an orthonormal eigenbasis, for every $\lambda\in \sigma_p(K_W)$, the  orthogonal projection $P_W(\lambda)$  on the $\lambda$-eigenspace of $K_W$ is defined as
$$P_W(\lambda) =\sum_{j\in\Z:~\lambda_j=\lambda} f_j\otimes f_j,$$
where $(f_j\otimes f_j)(f)= \langle f, f_j\rangle f_j$ and $\langle \cdot,\cdot\rangle$ is the inner product in $\cH$.
We now define the spectral measure corresponding to $K_W$ as 
$$
P_W(A) =\sum_{\lambda\in\sigma(K_W)\cap A} P_W(\lambda),\qquad
 A\in \cB([0,1]),
$$
where $\cB([0,1])$ stands for the Borel
$\sigma$-algebra of subsets of $[0,1]$, and the sum in the definition of $P_W(A)$, if infinite, is interpreted as convergence in the strong operator topology.
Clearly, the spectral measure $P_W$ is independent of the choice of the eigenbasis for $K_W$. 

\subsection{Vague convergence}\label{sec.vague}

Let $\cM$ stand for the set of spectral measures corresponding to $\{K_W:\; W\in\cW_0\}$.
On $\cM$ we define vague convergence as follows.

\begin{df}\label{def.vague} A net of spectral measures $\{P_\gamma\}_{\gamma\in I}\subseteq\cM$ is said to
  converge to $P\in\cM$ vaguely,  denoted
  $P_\gamma\stackrel{v}{\longrightarrow} P$, if there is a dense set $D\subset\R$
  such that for any $a, b\in D$
\begin{equation*}
  \lim_{\gamma\in I} P_\gamma \left((a,b]\right)=P\left((a,b]\right),
\end{equation*}
where the limit is in the operator norm.
\end{df}

The space of bounded Borel measurable functions on $\R$,
denoted by $B(\R)$, is a commutative C$^*$-algebra under the pointwise algebra operations, complex conjugation,
and uniform norm.
Since every $f\in B(\R)$ is a uniform limit of simple functions, for every $P\in\cM$ we can define $\int f dP$ as the limit (in the operator norm) of finite sums.
Moreover, the map $f\mapsto \int f dP$ is a $*$-homomorphism from $B(\R)$ to the C$^*$-algebra
$\cL(\cH)$ of bounded linear operators on $\cH$. As a result integration is contractive, i.e.,
$\|\int f dP\|\leq \|f\|_\sup$
(see Theorem 1.43 and Proposition 1.42(b) of \cite{1995:Folland:HarmonicAnalysis}).

The following theorem gives an equivalent characterization of vague convergence.
\begin{thm}\label{thm:def2-vague}
  $\{P_\gamma\}_{\gamma\in I}\subseteq \cM$ converges to  $P$ vaguely iff 
\begin{equation}\label{eq:vague-def2}
\lim_{\gamma\in I} \int_{\R} f \, dP_\gamma = \int_{\R} f \, dP \qquad \forall f\in C_c(\R),
\end{equation}
where the limit is taken with respect to the operator norm.
\end{thm}
\begin{proof}
Suppose $P_\gamma\stackrel{v}{\longrightarrow} P$.
Given $f\in C_c(\R)$ and  $\epsilon>0$, we can use the uniform continuity of $f$ together with the density of $D$ to find $\{a_i,b_i\}_{i=1}^k\subseteq D$ and constants $\alpha_i\in \R$ such that a linear combination of characteristic functions of $(a_i,b_i]$ provides an $\epsilon$-approximation of $f$, that is, 
$$\|f-\sum_{i=1}^k\alpha_i\chi_{(a_i,b_i]}\|_\sup<\epsilon/3.$$
Given that $P_\gamma((a_i,b_i])\to P((a_i,b_i])$ for $1\leq i\leq k$, and that integration against any measure in $\cM$ is a contractive operator, an $\epsilon/3$ argument gives \eqref{eq:vague-def2}.

Conversely, suppose that $P_\gamma\to P$ is a converging net in $\cM$ in the sense of \eqref{eq:vague-def2}. Suppose $P=P_W$ and $P_\gamma=P_{W_\gamma}$ for $W,W_\gamma\in\W_0$.
Let $D=\R\setminus \sigma(K_W)$. For $a\in D$, we can find $f\in C_c^+(\R)$ and $0<\delta_a<\rho_a<|a|$ such that $f$ is supported in $(a-\rho_a,a+\rho_a)\subseteq \R\setminus \sigma(K_W)$ and $f=1$ on $(a-\delta_a,a+\delta_a)$. Clearly, $\int fdP=0$, so by the convergence assumption, we get $\int f dP_\gamma\to 0$ in the operator norm. 
If for some $\gamma$, the set $\sigma(K_{W_\gamma})\cap (a-\delta_a,a+\delta_a)$ is nonempty, then $\|\int f dP_\gamma\|\geq1$. This follows from the fact that integration is a $*$-homomorphism, and we have
$$\left(\int f dP_\gamma\right)\left( \int \chi_{(a-\delta_a,a+\delta_a)}dP_\gamma\right)=\left(\int f \chi_{(a-\delta_a,a+\delta_a)}dP_\gamma\right)=\int \chi_{(a-\delta_a,a+\delta_a)}dP_\gamma,$$
and the final integral is a nonzero orthogonal projection on the eigenspaces  of $K_{W_\gamma}$ associated with $\lambda\in (a-\delta_a,a+\delta_a)$, so it has norm 1. Thus, 
$$1=\left\|\left(\int f dP_\gamma\right)\left( \int \chi_{(a-\delta_a,a+\delta_a)}dP_\gamma\right)\right\|\leq \left\|\int f dP_\gamma\right\|\left\|\int \chi_{(a-\delta_a,a+\delta_a)}dP_\gamma\right\|=\left\|\int f dP_\gamma\right\|.$$
This observation, combined with  $\int f dP_\gamma\to 0$, implies that $\sigma(K_{W_\gamma})\cap (a-\delta_a,a+\delta_a)=\emptyset$ holds eventually along the net.
Using this fact, we can now show that $P_\gamma\stackrel{v}{\longrightarrow} P$. Indeed, let $a<b$ be elements in  $D$, and suppose $\delta_a$ and $\delta_b$ are the positive numbers as described above. Let $f\in C_c(\R)$ be defined so that ${\rm supp}(f)\subseteq (a-\delta_a,b+\delta_b)$ and $f=1$ on $(a,b]$. Then, the above discussion implies that 
$$\int f dP=P((a,b]) \quad \mbox{and}\quad \int f dP_\gamma=P_\gamma((a,b]),$$
which finishes the proof.
\end{proof}

\begin{rem}\label{rem.topological}
  Using Theorem~\ref{thm:def2-vague}, it is easy to verify that the vague convergence satisfies the
  properties of a topological convergence, and can be used to induce a topology on $\cM$.  For a discussion on ``nets describe topologies'', see
  \cite[Section 11, Exercise 11.D]{willard}.
\end{rem}

\subsection{The cut distance and equivalence relations}\label{sec.equivalence}
Let $S$ (respectively $\bar S$) stand for the set of measure-preserving bijections (maps) from $[0,1]$ to itself.
For $W\in \cW$, $W^\phi(x,y)\doteq W\left(\phi(x),\phi(y)\right)$ is called a pullback if
$\phi\in\bar S$, or a rearrangement if $\phi\in S$. For $f\in\cH$, let 
$f^\phi$ denote $f\circ\phi$.

\begin{df}\lbl{df.cut}
  The cut distance on $\cW$ is defined by
  $$
  \delta_\Box (U,V)=\inf_{\phi\in S} \|U-V^\phi\|_\Box, \qquad U,V\in\cW.
  $$
\end{df}
The following equivalent characterizations of the cut distance are useful
(cf.~\cite[Theorem~8.13]{LovaszAMS}):
\begin{align*}
  \delta_\Box (U,V)& =\inf_{\phi\in S} \|U-V^\phi\|_\Box=\inf_{\phi\in \bar S} \|U-V^\phi\|_\Box\\
                   & = \inf_{\psi\in S} \|U^\psi-V\|_\Box=\inf_{\psi\in \bar S} \|U^\psi-V\|_\Box\\
   & = \inf_{\phi,\psi\in S} \|U^\psi-V^\phi\|_\Box=\min_{\phi, \psi\in \bar S} \|U^\psi-V^\phi\|_\Box.  
  \end{align*}

  Note that $U^\phi\neq U$ in general and thus the cut distance is only a pseudometric on $\cW$.
  To obtain a metric space, one identifies kernels with cut distance zero. Specifically, we introduce
  an equivalence relation on $\cW$:
  $$
  U\cong V \qquad \mbox{if}\qquad \delta_\Box(U,V)=0
  $$
  and define a quotient space $\widehat \cW=\cW/\cong$. The space $(\widehat \cW, \delta_\Box)$ is a metric space.
  Moreover, $(\widehat \cW_0, \delta_\Box)$ is a compact metric space (cf.~\cite{LovaszAMS}).

  Next, we define pullbacks and rearrangements for kernel operators and
  spectral measures. For the kernel operators, let
  $$
  K_W^\phi\doteq K_{W^\phi},\qquad \phi\in\bar S.
  $$
  
  \begin{lem}\label{lem:Kw-with phi}
For $\phi\in\overline{S}$ and $W\in \cW$, we have 
\begin{itemize}
\item[(i)] $ K_W^\phi f^\phi =(K_W f)^\phi$ for all $f\in \cH$.
\item[(ii)] $ K_W^\phi=0$ on the subspace $\Phi_\phi(\cH)^\perp$, where $\Phi_\phi:\cH\to \cH$ is defined as $\Phi_\phi(f)=f^\phi$. 
\end{itemize}
  \end{lem}
  \begin{proof}
  Part (i) is a direct consequence of the fact that $\phi$ is measure-preserving, and leaves integration invariant. So
  $$K_W^\phi f^\phi(x)=\int W^\phi(x,y)f^\phi(y)\, dy=\int W(\phi(x),\phi(y))f(\phi(y))\, dy=
  \int W(\phi(x),y)f(y)\, dy=(K_W f)^\phi(x).$$
  For part (ii), suppose $g\in \cH$ such that $\langle g, f^\phi\rangle=0$ for all $f\in \cH$. Then 
\begin{eqnarray*}
  \langle K_W^\phi g, f\rangle&=&\iint W^\phi(x,y)g(y)\overline{f(x)}\, dy dx
                                  =\int g(y)\overline{\left(\int W^\phi(x,y)f(x)\, dx\right)} \, dy\\
&=&\int g(y)\overline{h^{\phi}(y)} \, dy=0,
\end{eqnarray*}
where the new function $h$ is defined as $h(x)=\int W(\phi(z),x)f(z)dz$.
 \end{proof}
  
  Extending this logic to
  the space of bounded linear operators $\cL(\cH)$, for $B\in \cL(\cH)$ and  $\phi\in\bar S$,
  we define $B^\phi$ by the following equation
 \begin{equation}\label{pullB}
   B^\phi f^\phi\doteq (B f)^\phi,\quad  \mbox{ and } \quad  B^\phi=0\; \mbox{ on } \;\Phi_\phi(\cH)^\perp.
   %\{f^\phi: f\in \cH\}^\perp.
 \end{equation}

 \begin{lem} \label{lem:general pullback}
 Equation \eqref{pullB} uniquely defines a bounded linear operator $B^\phi$.
  \end{lem}
   \begin{proof} 
     Fix $\phi\in\overline{S}$, and recall that $\Phi_\phi:\cH\to \cH$ is defined as $\Phi_\phi(f)=f^\phi$.
     Clearly, $\Phi_\phi$ is a linear operator, and since $\phi$ is measure-preserving,
     the operator $\Phi_\phi$ is an isometry. 
  Therefore, the range $\Phi_\phi(\cH)$ is a closed subspace of $\cH$, and we have the orthogonal decomposition $\cH=\Phi_\phi(\cH)\oplus  \Phi_\phi(\cH)^\perp$. 
Given a bounded linear  operator $B\in \cL(\cH)$, the operator $B^\phi:\cH\to \cH$ is defined as follows
$$ B^\phi(\Phi_\phi (f))=\Phi_\phi(Bf), \forall f\in \cH,\quad \mbox{ and } \quad B^\phi=0 \mbox{ on }
\Phi_\phi(\cH)^\perp.$$
The linearity and boundedness of $B^\phi$ are clear. We only need to show that $B^\phi$ is well-defined: taking $f,g\in \cH$ with $f^\phi=g^\phi$, we need to show that $\Phi_\phi(Bf)=\Phi_\phi(Bg)$ as well. However, since $\Phi_\phi$ is an isometry, $f^\phi=g^\phi$ implies that $f=g$, which finishes the proof. 
   \end{proof}
   
% \begin{df}\label{def:delta-op}
% For $P, Q\in \mathcal{L}(\cH)$  we define 
%  \begin{equation*}
%   \delta_{\op} (P,Q) \doteq \inf_{\phi\in S} \| P^\phi -Q\|.
%  \end{equation*}  
% \end{df}
%\marginpar{I commented out def of $ \delta_{\op}$} Yes! No need for this def.

We now turn to spectral measures.
The corresponding equivalence relation on $\cM$ can be established using  a $1-1$ correspondence between
graphons and spectral measures:
    \begin{align*}
      W & \mapsto P_W,\\
      P_W&\mapsto W=\sum_{\lambda_i\in \sigma(K_W)} \lambda_i P_W(\{\lambda_i\}).
    \end{align*}
    Thus, one can define $P\cong Q$ if $\exists \, W, V\in \W_0:\; P=P_W$ and  $Q=P_V$ such that  $\delta_\Box (W,V)=0$.
    We can now define the quotient space $\widehat\cM=\cM/\cong$.

Let $F:\W_0\to \cM$ be the map $W\mapsto P_W$, and define $\widehat F:\widehat\cW_0\to\widehat\cM$ as the map induced on the corresponding quotient spaces. 
Clearly, the map $\widehat{F}$ is well-defined.

The following proposition clarifies how the equivalence relation on $\cM$ relates to measure-preserving maps in a more direct way.
\begin{proposition}\label{prop:M-equiv}
For every $W\in \W_0$ and $\phi\in\bar{S}$, we have $P_{W^\phi}(A)=(P_W(A))^\phi$ for every $A\in\cB([0,1])$ \red{with $0\not\in A$. 
Moreover, $P_{W^\phi}(\{0\})$ is just the sum of $(P_{W}(\{0\}))^{\phi}$ and the orthogonal projection on $\Phi_\phi(\cH)$.} 
Consequently, for every $P,Q\in \cM$, we have $P\cong Q$ if and only if there exists
$\phi,\psi\in \bar{S}$ such that $(P(A))^\phi=(Q(A))^\psi$ for every $A\in\cB([0,1])$.
\end{proposition}
\begin{proof}
For a graphon $W\in\cW_0$, let $\{f_j\}_j$ be an orthonormal basis for $\cH$ of eigenvectors of $K_W$ associated with eigenvalues $\{\lambda_j\}$.  By Lemma~\ref{lem:Kw-with phi} (i), 
$$K_W^\phi(f_j^\phi)=(K_Wf_j)^{\phi}=\lambda_jf_j^\phi.$$
So $\{f_j^\phi\}_j$ forms a set of eigenvectors of $K_{W^\phi}$ associated with $\lambda_j$. Since $\Phi_\phi$ is an
isometry, the set $\{f_j^\phi\}_j$ is an orthonormal set as well.  By Lemma~\ref{lem:Kw-with phi}, every
$\lambda$-eigenvector of $K_{W}^\phi$ with $\lambda\neq 0$ must be of the form $f^\phi$ for some $f\in \cH$.
Let $f=\sum_{j\in \Z} \alpha_j f_j$ be the expansion of $f$ with respect to the basis $\{f_j\}_j$. Then, $f^\phi=\sum_{j\in \Z} \alpha_j {f_j}^{\phi}$. Thus, 
$\{f_j^\phi\}_j$  contains an orthonormal basis for every eigenspace of $K_W^\phi$ associated with nonzero eigenvalues. So,
$$P_{W^\phi}(\lambda) =\sum_{j:~\lambda_j=\lambda}  f_j^\phi \otimes  f_j^\phi, \quad \mbox{ when }\lambda\neq 0.$$
It is easy to see that $(f_j\otimes f_j)^\phi=f_j^\phi \otimes  f_j^\phi$ (by equation \eqref{pullB} for example). So $P_{W^\phi}(\lambda)=(P_{W}(\lambda))^{\phi}$ when $\lambda\neq 0$.

\red{To deal with the case $\lambda=0$, note that if $K_W^\phi(f^\phi)=0$, then $(K_Wf)^{\phi}=0$, which implies that $K_Wf=0$.
Using this fact, together with Lemma~\ref{lem:Kw-with phi} (ii), we have $P_{W^\phi}(0)=Q_1 + Q_2$ where $Q_1$ and $Q_2$ are the orthogonal projections on 
the closed subspaces $\Phi_\phi(\ker(K_W))$ and  $\Phi_\phi(\cH)^\perp$ respectively. Using a similar argument, it is easy to verify that $Q_1=(P_{W}(0))^\phi$.
Therefore, the spectral measure associated with $K_{W}^{\phi}$ is given as follows. For  $A\in\cB([0,1])$ with $0\not\in A$, 
$$P_{W^\phi}(A) =\sum_{\lambda\in\sigma(K_W^\phi)\cap A} P_{W^\phi}(\lambda)=\sum_{\lambda\in\sigma(K_W)\cap A} P_{W^\phi}(\lambda)=\sum_{\lambda\in\sigma(K_W)\cap A} (P_W(\lambda))^\phi=(P_{W}(A))^{\phi},$$
and for $A\in\cB([0,1])$ with $0\in A$, we have $P_{W^\phi}(A) =(P_{W}(A))^{\phi}+Q_2$.
This proves the first part of the proposition.}

Next, suppose  $P=P_W$ and $Q=P_V$ are elements of $\cM$ with $W,V\in \W_0$. By definition, $P\cong Q$ holds precisely when $\delta_\Box(W,V)=0$, which in turn means $W^\phi=V^\psi$ for some $\phi,\psi\in \bar{S}$. Appealing to the first part of this proposition, for every $A\in\cB([0,1])$ with $0\not\in A$, we have
$$(P(A))^\phi=(P_W(A))^\phi=P_{W^\phi}(A)=P_{V^\psi}(A)= (P_V(A))^\psi=(Q(A))^\psi.$$
\red{On the other hand, we have 
$$(P(\{0\}))^\phi=(P_W(\{0\}))^\phi=P_{W^\phi}(\{0\})-Q_2=P_{V^\psi}(\{0\})-Q_2= (P_V(\{0\}))^\psi=(Q(\{0\}))^\psi.$$
This proves the ``only if'' direction of the statement. To prove the ``if'' direction, suppose $(P(A))^\phi=(Q(A))^\psi$ for every $A\in\cB([0,1])$. With a similar argument as before, this assumption implies that $P_{W^\phi}(\{\lambda\})=P_{W^\psi}(\{\lambda\})$ for every $\lambda\in\bbR$. We then use  the $1-1$ correspondence between kernels and spectral measures to conclude that $W^\phi=V^\psi$, which finishes the proof.}
\end{proof}

\section{Continuous dependence}\label{sec.continuous}
\setcounter{equation}{0}

In this section, we prove continuous dependence of spectral measures on the corresponding
graphons. To this end, we  need to specify a topology on  $\widehat\cM$.
This is done by extending the notion of vague convergence to the quotient   space.

    \begin{df}\label{def.hat-vague}
   A net   $\{\widehat{P}_\gamma\}_{\gamma\in I}\subseteq \widehat\cM$  is said to converge to
  $\widehat P\in\widehat{\cM}$ vaguely,  if there exist \red{$W_\gamma\in \cW_0$,} $\phi_\gamma\in S$  and a dense set $D\subset\R$
  such that \red{$P_\gamma=P_{W_\gamma}$ and} for any $a, b\in D$
  $$
  \lim_{\gamma\in I}\red{P_{{W_\gamma}^{\phi_\gamma}}}(a,b]= P(a,b], 
  $$
  with limit taken in operator norm.
\end{df}

\begin{thm}\label{thm.cont}
$\widehat F:\widehat\cW\to\widehat\cM$ is continuous.
\end{thm}

The proof of the theorem relies on the following lemma, which is of independent interest.

\begin{lem}\lbl{lem.converge}
Suppose $\{W_n\}$ is a sequence of graphons from $\cW_0$ converging to $W$ in the cut norm.
Let $K_{W_n}f=\int W_n(\cdot,y)f(y)dy$ and denote the eigenvalues, normalized
eigenvectors and the corresponding spectral measure of $K_{W_n}$
by $\{\lambda_{j}^n\}$, $\{f_{n,j}\}$, and $P_n:=P_{W_n}$ respectively. In addition, let
$P:=P_{W}$.

Then for any $\red{0\neq} a<b\le \infty$ such that $\{a,b\}\bigcap \sigma(K_W)=\emptyset$,
  $P_n\left((a,b]\right)$ tends to $P\left((a,b]\right)$ in the operator norm.
\end{lem}

We first prove Theorem~\ref{thm.cont} using Lemma~\ref{lem.converge}. Then we give the proof
of the lemma.

\begin{proof}[Proof of Theorem~\ref{thm.cont}] 
Since $(\widehat\cW,\delta_\Box)$ is a metric space, it is enough to check the convergence on sequences, rather than nets.
Suppose a sequence $\{W_n\}_{n\in \N}$ converges to $W$ in $\widehat\cW$, i.e., $\delta_\Box(W_n\,W)\to 0$ as $n$ tends to $\infty$. By definition of $\delta_\Box$, we can choose 
$\phi_n\in S$ such that $\|W_n^{\phi_n}-W\|_\Box\leq 2\delta_\Box(W_n,W)$, so we have $\|W_n^{\phi_n}-W\|_\Box\to 0$. Applying Lemma~\ref{lem.converge} to the converging sequence $\{W_n^{\phi_n}\}$, we conclude that the sequence $\{P_{W_n^{\phi_n}}((a,b])\}$ converges to $P_W((a,b])$ in the operator norm for every $a,b\in \R\setminus\sigma(K_W)$. \red{Note that we always have $0\in\sigma(K_W)$, so the condition $a\neq 0$ is satisfied.} Setting $D=\R\setminus\sigma(K_W)$, we see that $\{\widehat{F}({W_n})\}$ converges to $\widehat{F}(W)$ vaguely.
 \end{proof}

\begin{proof}[Proof of Lemma~\ref{lem.converge}] 
Suppose $0<a<b\le \infty$ and $\{a,b\}\bigcap \sigma(K_W)=\emptyset$. Since $0$ is the
    only possible accumulation point of $\sigma(K_W)$, at most finite number of eigenvalues
    of $K_W$ lie in $(a,b)$. Assume that
    $$
    (a,b)\bigcap \sigma(K_W)=\left\{ \lambda_{k+1}, \lambda_{k+2},\dots, \lambda_{k+p}\right\}
    $$
    where $b>\lambda_{k+1}\ge \lambda_{k+2}\ge\dots\ge \lambda_{k+p}>a$ for some
    nonnegative integers $k$ and $p$.
Next, since $\lim_n\lambda^n_j=\lambda_j,\; \;j\in \dot\Z$ (cf.~\cite{Sze-2011}), there are $\delta>0$
    and $n_1\in\N$ such that
    \begin{align*}
      \lambda^n_j \in (a,b) & \quad\mbox{for}\quad j\in J\doteq\{k+1, k+2,\dots, k+p\},\quad n> n_1,\\
      \lambda^n_j \notin (a-\delta,b+\delta) & \quad\mbox{for}\quad j\in\dot\Z\setminus J,\quad n> n_1.
    \end{align*}
For every $j\in J$, going down to a subsequence if necessary, $f_{n,j}$  converge weakly to some $\tilde f_j$.
    By \cite[Lemma~1.10]{Sze-2011}, $\tilde f_j$ is a normalized eigenvector of $K_W$ corresponding
    to $\lambda_j$, and $f_{n,j}$  converge  to $\tilde f_{j}$ in $L^2$.
    Since $\langle f_{n,j}, f_{n,i}\rangle =0$ for every $n\in\N$ and every pair $i\neq j$,
    $i,j \in J$, the vectors $(\tilde f_j)_{ j\in J}$ are mutually orthogonal. 
    %
%From the above argument, we have that every subsequence of $\{f_{n,j}\}$ contains a subsequence converging to $\tilde{f}_j$, which implies that for all $j\in J$, we have $f_{n,j}\to \tilde{f}_j$ as $n\to \infty$.
  So, $P_n\left((a,b]\right)$ tends to $P\left((a,b]\right)$ in the operator norm.
    The case $(a,b)\bigcap \sigma (K_W)=\emptyset$ is treated similarly. Likewise, the arguments
    above apply verbatim to the case when $-\infty\leq a<b<0$.
Finally, suppose $a<0<b\le \infty$ and $\{a,b\}\bigcap \sigma(K_W)=\emptyset$. In this case, note
    $$
    P_n \left((a,b]\right) = I-P_n \left((-\infty,a]\right) - P_n \left((b,\infty]\right)
    \longrightarrow
    I- P \left((-\infty,a]\right) - P \left((b,\infty]\right)= P \left((a,b]\right).
    $$
  \end{proof}

 \section{Large deviations}
 \label{sec.ldp}
\setcounter{equation}{0}
In this section, we review an LDP for $W$-random graphs established in \cite{DupMed22} and then
extend it to $W$-random spectral measures.

For a given $W\in\cW_0$, recall that $W_n$ (cf. \eqref{Wn}) denotes a random graphon  corresponding to a $W$-random
graph $\Gamma_{W_n}$ defined in \eqref{def-Wn-1}, \eqref{def-Wn-2}.
The $W$-random graphon $W_n$ defines the integral operator $K_{W_n}$ (cf. \eqref{int-operator})
and  the corresponding spectral measure $P_n=P_{W_n}$. The spectral measure $P_{W_n}$ is called  a $W$-random spectral measure.
Let $(\lambda^n_j)$ (resp.~$(\lambda_j)$) denote the eigenvalues of $K_{W_n}$ (resp.~$K_W$).
We have $W_n\to W$ in the cut norm almost surely (a.s.) (cf.~\cite{DupMed22}). This implies
\begin{align}\lbl{eig-converge}
  &\lim_{n\to\infty} \lambda^n_j=\lambda_j, \quad j\in\dot\Z,& \quad (\mbox{cf.~\cite{Sze-2011}}),\\
\lbl{Pn-converge}
  &P_{n}\stackrel{v}{\longrightarrow} P_W,&          \quad (\mbox{cf.~Lemma~\ref{lem.converge}}).
   \end{align}                                          

\subsection{The LDPs for random graphs and spectral measures}   \label{sec.ldp-old}
For a given $W\in\cW$, and for $\widehat V\in \widehat{\cW}$ let  
\begin{equation}\label{rate}
I(\widehat V)=\inf_{V\in\widehat{V}}\Upsilon(V,W),
\end{equation}
where $\Upsilon$ is defined by 
\begin{equation}\label{Ups}
  \begin{split}
  \Upsilon(V,W) &=\frac{1}{2}\int_{[0,1]^2}  R \left( \{ V(x,y), 1-V(x,y) \} \Vert \{ W(x,y), 1-W(x,y) \}  \right) dxdy\\
&=\frac{1}{2}\int_{[0,1]^2} \left\{  V(x,y) \log\left( {V(x,y)\over W(x,y)}\right) +
  \left( 1-V(x,y) \right) \log\left( {1-V(x,y)\over 1-W(x,y)} \right) \right\} dxdy,
\end{split}
\end{equation}
and $R(\theta\Vert\mu)$ is the relative entropy of probability measures $\theta$ and $\mu$, i.e., 
\begin{equation*}
R\left( \theta\left\Vert \mu\right.\right)=\int \left(\log {\frac{d\theta}{%
d\mu }}\right)d\theta
\end{equation*}
if $\theta \ll \mu$ and $R\left( \theta\left\Vert \mu\right.\right)=\infty$ otherwise.

\begin{thm} \cite[Theorem 4.1]{DupMed22}\;
  \label{thm.Wrandom}
Let $\{W^n\}$  be defined by \eqref{Wn}. Then
  $\{\widehat{W}^{n}\}_{n\in\mathbb{N}}$ 
satisfies the LDP
with scaling sequence $\frac{n^2}{2}$ and rate function \eqref{rate}. 
In particular, the function $I$ has compact level sets on $\widehat \cW$,
\[
\liminf_{n\rightarrow \infty}\frac{2}{n^2} \log P\{\widehat{W}^{n}\in O\} \geq - \inf_{\widehat{V} \in O}I(\widehat{V})
\]
for open $O \subset \widehat \cW$,
and 
\[
\limsup_{n\rightarrow \infty}\frac{2}{n^2} \log P\{\widehat{W}^{n}\in F\} \leq - \inf_{\widehat{V} \in F}I(\widehat{V})
\]
for closed $F \subset \widehat \cW$.
\end{thm}

To translate the LDP for $W$-random graphs to that for spectral measures, we use the
Contraction Principle \cite[Theorem~4.2.1]{Dembo-Zeitouni}. 
To this end, recall that \red{$F: \cW_0\to\cM$} establishes the correspondence between graphons and spectral
measures, i.e.,
\red{$F:~\cW_0\ni W\,\mapsto\, P_W\in \cM.$}
The map $F$ induces
\red{$$\widehat F:~\widehat \cW_0\ni \widehat W\,\mapsto\, \widehat P_W\in \widehat\cM$$}
on the relevant quotient spaces, which is continuous by Theorem~\ref{thm.cont}.
%
%Clearly, all members of the equivalence class $\widehat W$ are mapped to the same equivalence class
%$\widehat{P_W}$. Thus,
%
%is well-defined. By Theorem~\ref{thm.cont}, $\widehat F$ is continuous. 
%
Thus, by the Contraction Principle, Theorem~\ref{thm.Wrandom} yields the following LDP for spectral measures.

\begin{thm}\label{thm.spectral-LDP} \label{sec.ldp-new}
  For \red{$W\in\cW_0$}, let $\{W^n\}$ be a sequence of $W$-random graphons.
  Consider the corresponding sequence 
 % of the eigenvalues $(\lambda^n_j)\in \ell^\infty$ and that
  of spectral measures $\widehat{\bP}^n\in \widehat\cM$.

  Then $\{\widehat{\bP}^n\}$ satisfies the LDP on  $\widehat\cM$
  with scaling sequence $\frac{n^2}{2}$ and the rate function
  $$
  J(\widehat P) =
  \inf\left\{ I(\widehat W):\; \widehat W\in\widehat\cW,\;  
  \widehat P=\red{\widehat{F}(\widehat W)}\right\}.
   $$
  % Likewise for every $j\in \dot\Z$, $(\lambda_j^n)$
  %  satisfies the LDP with the scaling sequence $\frac{n^2}{2}$ and the rate function
  %  $$
  %  J_j^2\left(    (\lambda^n_j)  \right) =\inf\left\{ I(\widehat W):\quad   (\lambda^n_j) =
   %                                       \widehat F_j^2\left(\widehat W\right)\right\}.
    %                                       $$
\end{thm}

\section{Examples}\label{sec.example}
\setcounter{equation}{0}

In this section, we estimate the eigenvalues of random kernel operators for small-world
and random bi-partite graphs subject to large deviations. To this end, we consider the
situation when the realizations of these graphs exhibit atypical edge distribution. For
this case, we use the explicit formula for the minimizer of the constrained rate function
derived in \cite[Lemma~7.2]{DupMed22}. The examples of this section are meant to illustrate potential
effects of the large deviations on the spectrum of $W$-random kernel operators.

Let $W\in \cW_0$ and consider $W^n\in\mathbb{G}(n,W)$. The number of undirected edges
in $W^n$ is given by
$$
\left| E(W^n)\right| = \frac{n^2}{2}\mathcal{L}(W^n),
\qquad  \mathcal{L}(W^n)\doteq\int_{[0,1]^2}W^n(x,y) \, dx dy.
$$
Since $\mathcal{L}(U)=\mathcal{L}(U^\phi)$ for any $\phi\in \bar S$, $\mathcal{L}$ is a functional
on $\widehat\cW_0$. By the Strong Law of Large Numbers, for $W^n\in\mathbb{G}(n,W)$ we have
$$
2n^{-2} \left| E(W^n)\right| \rightarrow w\doteq\mathcal{L}(W), \quad 
n\to\infty\quad \mbox{almost surely}.
$$

Thus, for a typical realization of $W^n$, $\mathcal{L}(W^n)$ lies in a vicinity of $w$. 
Fix
$\delta>0$ and consider a rare event\footnote{The case $\delta<0$ is analyzed similarly.}
$$
A_{n,\delta}=\left\{\mathcal{L}(W^n)\geq w+\delta\right\}.
$$
By \cite[Theorem~7.1]{DupMed22},
$$
\mathbb{P}\left(\delta_\Box \left(\widehat W_n, \widehat W^\ast_\delta\right)
\geq\epsilon \left| A_{n,\delta} \right.\right)
\le \exp\{ -C(\epsilon,\delta) n^2\}.
$$
Here, $W^\ast_\delta$ is a minimizer of the rate function restricted to $A_{n,\delta}$.
For the model at hand, the constrained optimization problem has an explicit solution
(cf.~\cite[Lemma 7.2]{DupMed22}):
$$
 W^\ast_\delta=\frac{W\xi}{1-W+W\xi},
$$
where $\xi\in (0,1)$ is a unique solution to 
\begin{equation}\label{constraint}
\int_{[0,1]^2} \frac{\xi W(x,y)}{1-W(x,y)+\xi W(x,y)}\, dxdy=w+\delta.
\end{equation}

The following lemma yields a leading order asymptotic formula for the solution of \eqref{constraint}.

\begin{lem}\label{lem.IFT} Let $W\in\cW_0$ and 
let $\delta>0$ be sufficiently small. Then
  \begin{equation}\label{ift}
    W^\ast_\delta = W+ \delta \frac{\ell(W)}{\int_{[0,1]^2} \ell(W)(x,y)\, dxdy} +O(\delta^2),\qquad
    \ell(W) \doteq W(1-W),
  \end{equation}
  provided
  \begin{equation}\label{ift-condition}
    \int_{[0,1]^2} \ell(W)(x,y)\, dxdy\neq 0.
  \end{equation}
\end{lem}
\begin{proof}
  Equation~\eqref{ift} follows from the Implicit Function Theorem. Specifically,
  let
  $$
  F(\eta)=\int_{[0,1]^2} \frac{\eta W(x,y)}{1-W(x,y)+\eta W(x,y)}\,  dxdy.
  $$
  Note
  $$
  F(1)=w,\qquad F^\prime (1)=\int_{[0,1]^2}\ell(W)(x,y) \, dxdy.
  $$
  By the Implicit Function Theorem, for $\delta\in (-\varepsilon,\varepsilon)$ there
  is a smooth function $\alpha(\delta)$ such that
  $$
  F\left(1+\alpha(\delta)\right)=w+\delta, \qquad \alpha(0)=0.
  $$
  Furthermore,
 \begin{equation}\label{found-alpha}
  \alpha(\delta)=\frac{\delta}{\int_{[0,1]^2}\ell(W)(x,y) \, dxdy} +O(\delta^2).
\end{equation}

%By plugging in $\xi = 1+\alpha$ into \eqref{constraint}, expanding the left--hand side
\red{By plugging in $\xi = 1+\alpha$ into the formula of $W_\delta^*$, expanding the right--hand side}
in $\alpha$ and using \eqref{found-alpha}, we derive \eqref{ift}.
\end{proof}

With Lemma~\ref{lem.IFT} in hand, we are in a position to discuss our two examples:
the small--world graphs and random bipartite graphs.

\subsection{Bipartite graphs}
Let $\alpha\in (0,1/2)$ 
and $p\in (0,1)$ and consider
\begin{equation}\label{bipart}
  W(x,y)=p\left[ \1_Q(x,y) + \1_Q(y,x)\right],\qquad Q\doteq [0,\alpha]\times [\alpha, 1].
\end{equation}

$K_W$ has two nonzero simple eigenvalues $\lambda_{1,2}=\pm\sqrt{\alpha (1-\alpha)}p$ with the corresponding
eigenfunctions
$$
v_{1,2}=\sqrt{1-\alpha} \1_{[0,\alpha)}(x) \pm\sqrt{\alpha} \1_{[\alpha, 1]}(x).
$$

In addition, $K_w$ has a zero eigenvalue $\lambda_0=0$. The corresponding eigenspace is 
the subspace
$$
V_0=\left\{f:\quad \int_{[0,\alpha)} f \, dx=0 \;\mbox{and}\; \int_{[\alpha, 1]} f \, dx=0.\right\}. 
$$

\red{From Lemma~\ref{lem.IFT},} under large deviations, $W$ is transformed to
$$
W^\delta (x,y)= \left(p +\frac{\delta}{2\alpha \left(1-\alpha\right)}\right) \left[ \1_Q(x,y) + \1_Q(y,x) \right] +O(\delta^2),
$$
 and the first two nonzero eigenvalues get perturbed to
    $$
    \lambda_{1,2}^\delta=\pm\left( \sqrt{\alpha (1-\alpha)}p +\frac{\delta}{2\sqrt{\alpha (1-\alpha)}} \right)+O(\delta^2).
    $$

    In addition to the eigenvalues of $K_W$, in applications one might be
    interested in  the eigenvalues of the graph Laplacian 
    $$
    (L_W f)(x)=\int W(x,y)\left( f(y)-f(x)\right) dy
    $$
    (see, e.g., \cite{MedPel23}).
    Below we compute the leading order approximations for the eigenvalues of $L_W$
    for the bipartite graph.

    The eigenvalues and the corresponding eigenvectors of $L_W$ can be easily computed
    \begin{description}
    \item[(i)] $\lambda_0=0$ and $v_0\equiv 1$,
    \item[(ii)] $\lambda_1=-p$ and $v_1=(\alpha-1)\1_{[0,\alpha)}(x) + \alpha \1_{[\alpha, 1]}(x)$, 
    \item[(iii)] $\lambda_2=-p\alpha$ with the  eigenspace
      $V_2=\left\{f(x)=f_0(x)\1_{[\alpha, 1]}(x)\;\&\; \int_\alpha^1 f_0(x)dx =0\right\}$.
      \item[(iiv)] $\lambda_3=-p(1-\alpha)$ with the  eigenspace
      $V_3=\left\{f(x)=f_0(x)\1_{[0, \alpha]}(x)\;\&\; \int_0^\alpha f_0(x)dx =0\right\}$.
      \end{description}

      The perturbed eigenvalues are
        \begin{align*}
        \lambda_1^\delta &=\lambda_1 -\frac{\delta}{2\alpha(1-\alpha)} + O(\delta^2),\\
        \lambda_2^\delta &=\lambda_2 -\frac{\delta}{2(1-\alpha)}+ O(\delta^2),\\
        \lambda_3^\delta &=\lambda_3 -\frac{\delta}{2\alpha}+ O(\delta^2).
      \end{align*}    
      Note that different eigenvalues are perturbed differently, with $\lambda_1$ getting the largest
      perturbation. Thus,
      the order of eigenvalues may change under large deviations. The corresponding eigenspaces,
      on the other hand, 
      remain unaffected by large deviations. We refer to this scenario as eigenvalue 
      switching (see \cite{Masioti2022} for  a related phenomenon
      in Principal Component Analysis).

      The eigenvalue switching may have important implications in bifurcation 
      problems, Principal Component Analysis, or Graph Signal Processing. Specifically,
      in nonlinear spatially extended dynamical systems,  at a Turing bifurcation,
      an emerging spatiotemporal pattern appears following 
      the loss of stability of spatially homogeneous solution.
      Such pattern is determined
      by the eigenfunctions corresponding to the small positive eigenvalue, which
      in network models is determined by the first nonzero eigenvalue of $K_W$
      or $L_W$ depending on the model at hand (see, e.g., \cite{MedPel23}).
      In this situation, the eigenvalue switching may result in a qualitatively
      different pattern, because the eigenspace corresponding to the principal
      eigenvalue has changed due to the new ordering of the eigenvalues.

      Likewise, in the context of Principal Component Analysis or Graph Signal 
      Processing, the importance of the projection of a function (data) on a 
      given eigenspace is determined by the order of the corresponding 
      eigenvalue. Thus, the eigenvalue switching will change the hierarchy of 
      the spectral projections in these problems.
      
      \subsection{Small--world graphs}\label{subsec:small-world}

Let $K_{q,p;r}(x)$ be a $1$--periodic function on $\R$, which is defined on $[-1/2, 1/2)$ by 
$$
K_{q,p;r}(x)=\left\{\begin{array}{ll}
               q, & |x|\le r,\\
               p,&  \mbox{otherwise}.
             \end{array}
           \right.
$$
Here, $r\in (0, 0.5]$ and $p,q\in (0,1)$. Consider $W(x,y)= K_{q,p;r}(x-y)$ and
$$
             (K_Wf)(x)= (K_{q,p;r}\ast f) (x)=\int_\bbT K_{q,p;r}(x-y) f(y) dy.
$$
Using the properties of the Fourier transform of a convolution, it is straightforward
to calculate the eigenvalues of $K_W$:
\begin{align}\label{lambda-0}
  \mu_0& = 2rq+(1-2r)p,\\
  \label{lambda-k}
  \mu_k&= \frac{q-p}{\pi k} \sin\left(2\pi kr\right), \quad k\in\dot\bbZ.
\end{align}
The corresponding eigenfunctions are $v_k=e^{\iu 2\pi kx},\; k\in\Z$.  
Note that
$\mu_0$ is a simple eigenvalue, while all other eigenvalues have
multiplicity at least $2$, because $\mu_k=\mu_{-k}$ for $k\neq 0$.

We now turn to $W(x,y)=K_{1-p,p;r}(x-y)$. Using \eqref{ift}, we find that
$$
W^\ast_\delta=W+\delta +O(\delta^2)=
K_{1-p+\delta +O(\delta^2), p+\delta+ O(\delta^2);r}(x-y).
$$
From this, using \eqref{lambda-0} and \eqref{lambda-k}, we have
\begin{align}\label{lambda-0-delta}
  \mu_0^\delta& = \mu_0+\delta +O(\delta^2),\\
  \label{lambda-k-delta}
  \mu_k^\delta&= \mu_k+\frac{O(\delta^2)}{\pi k} \sin\left(2\pi kr\right), \quad k\in\dot\bbZ.
\end{align}
This shows explicitly how the eigenvalues of the small-world graphons are perturbed.
Note in particular that the perturbation affects the principal eigenvalue $\mu_0$ the
most. The effect on all other eigenvalues is smaller and it becomes even smaller for larger
$|k|$.

% \subsection{Other examples}
% \begin{itemize}
% \item For finite rank graphons it would
%   be intersting to see what happens to $\ker K_W$ of infinite co-dimension.
%   Consider $W=\cos(2\pi x)$. This is a rank-$2$ graphon.
%   Lemma~\ref{lem.IFT} does not apply, because \eqref{ift-condition} fails. We may
%  want to look at higher order terms to approximate $W^\ast_\delta$.
% \item Check if one can use interlacing theorem to make general statements about the
%  eigenvalues of the perturbed (finite-rank) graphons.
% \item Other examples: bipartite graphs.
% \item For $W(x,y)=\exp\{-\kappa (x-y)\}$ the eigenvalues can be computed explicitly
%  (see Example~5.5 in \cite{ChiMed19a}).
%  \end{itemize}

\section{Discussion}\label{sec.discuss}
\setcounter{equation}{0}

Motivated originally by problems in combinatorics and theoretical computer science \cite{LovaszAMS}, the theory of
graphons
rapidly penetrated different areas of mathematics, as well as its applications to physics, biology, and other disciplines.
Examples include random graphs in probability \cite{Cha16}, interacting
particle systems  in statistical physics \cite{OliRei20, DupMed22}, coupled dynamical systems
\cite{Med14b, KVMed18},
mean field games in economics and control systems \cite{CaiHua21, CaiHo22}, and signal processing in
engineering \cite{Ruiz-2021, GhaJan22},
to name a few. Graphons provide an effective analytically  tractable way for modeling large networks of different nature.
They opened a way for rigorous treatment of network models, which were outside the realm of mathematical analysis before.

The  analysis  of different network models often relies on the  information about eigenvalues and the 
eigenvectors of the kernel operators defined by graphons.
% For instance, stability and bifurcations of synchrony and
% other regimes in dynamical networks requires information about the leading eigenvalue of the kernel operator.
Spectral methods feature prominently in the analysis of synchronization and pattern formation in dynamical networks \cite{ChiMed19a, MedPel23}
and in signal processing on graphs \cite{GhaJan22}. In the analysis of dynamical models the eigenvalues
and the eigenvectors are
used to identify bifurcations and to describe the bifurcating solutions. In signal processing, spectral projections are
used in the analysis of network data.

When dealing with random networks, one has to take into account rare events when the structure of the network
deviates from typical realizations. Such events, although extremely rare, are inevitable for sufficiently many realizations
of the random graph model. The LDP quantifies the probability of such events on the one hand.
On the other hand, it describes a typical realization of such rare events. The latter may be used to describe the system
outcome, whether it is in the form of network dynamics or data projections depending on the model at hand. 

The LDP for spectral measures identified and proved in this paper, describes how the eigenvalues and
the eigenvectors change under large deviations. The examples analyzed in Section~\ref{sec.example} suggest
several qualitative phenomena, which have a potential to be useful in applications. First, the eigenvalue estimates derived
in this section show that the large deviations in the structure of the $W$-random graph affect the principal eigenvalue
the most. This may lead to notable changes in the spectral gap, which is important for stability
of certain dynamical regimes and for the rate of convergence of the random walk models. In addition,
since different eigenvalues are affected differently, while the eigenspaces of the kernel operators
based on Cayley graphons or stochastic block graphons remain unaffected, the switching of the order of spectral
projections may take place. This may have implications in bifurcation problems, the principal component analysis, and signal processing on graphs, to name a few potential areas of applications.

\section*{Acknowledgements}
This work was partially supported by the National Science Foundation Grants DMS-1902301 (to MG) and DMS-2009233 (to GM).

\bibliographystyle{amsplain}
%\bibliography{signal-combined}
\def\cprime{$'$} \def\cprime{$'$} \def\cprime{$'$}
\providecommand{\bysame}{\leavevmode\hbox to3em{\hrulefill}\thinspace}
\providecommand{\MR}{\relax\ifhmode\unskip\space\fi MR }
\providecommand{\MRhref}[2]{%
  \href{http://www.ams.org/mathscinet-getitem?mr=#1}{#2}
}
\providecommand{\href}[2]{#2}

\end{document}